\documentclass{tsp-short}

\usepackage{amsfonts}
\usepackage{amssymb}
\usepackage{newlfont}
\usepackage{amsthm}
\usepackage{euscript}
\usepackage{amsmath}
\usepackage{graphicx}
\usepackage{dsfont}
\usepackage{enumerate}

\theoremstyle{plain}
\newtheorem{theorem}{Theorem}
\newtheorem{corollary}{Corollary}
\newtheorem{lemma}{Lemma}
\newtheorem{proposition}{Proposition}
\theoremstyle{definition}
\newtheorem{definition}{Definition}
\theoremstyle{remark}
\newtheorem{remark}{Remark}

\newcommand{\Var}{\mathop \mathrm{Var}}
\newcommand{\supp}{\mathop \mathrm{supp}}
\newcommand{\1}{\mathds{1}}

\newcommand{\wt}{\widetilde}

\newcommand{\pt}{\partial}
\newcommand{\ve}{\varepsilon}
\newcommand{\vf}{\varphi}

\newcommand{\mbR}{{\mathds R}}

\newcommand{\mbQ}{{\mathds Q}}

\newcommand{\cF}{{\cal F}}

\newcommand{\cB}{{\cal B}}

\begin{document}

\title[On differentiability of a solution of an SDE]{On differentiability  with respect to the initial
data of a solution of an SDE with L\'evy noise and discontinuous
coefficients}

\author{Olga V. Aryasova}
\address{Institute of Geophysics, National Academy of Sciences of Ukraine,
Palladin pr. 32, 03680, Kiev-142, Ukraine}
\email{oaryasova@mail.ru}
\author{Andrey Yu. Pilipenko}
\address{Institute of Mathematics,  National Academy of Sciences of
Ukraine, Tereshchenkivska str. 3, 01601, Kiev, Ukraine}
\email{pilipenko.ay@yandex.ua}

 \dedicatory{}


\begin{abstract} We construct a stochastic flow generated by an SDE with its drift being a function of bounded variation and its noise being a stable process with exponent from (1,2). It is proved that the flow is non-coalescing and Sobolev differentiable with respect to initial data. The representation for the derivative is given.
\end{abstract}
\maketitle \thispagestyle{empty}
\section{Introduction}
Consider an SDE
\begin{equation}
\label{eq1.1} \vf_t(x)=x+\int^t_0a(\vf_s(x))ds+Z(t), t\geq0,
\end{equation}
where $x\in\mathds{R}$, $a$ is a bounded measurable function on $\mathds{R}$, $(Z(t))_{t\geq0}$ is
a symmetric stable process with the exponent $\alpha\in(1, 2),$ i.e. $(Z(t))_{t\geq0}$ is a L\'evy process with its characteristic
function being equal to
$$
E\exp\{i\lambda Z(t)\}=\exp\{-ct|\lambda|^\alpha\}, \lambda\in\mbR,
$$
where $c>0$ is a constant.

The existence and uniqueness of a weak solution, and the strong Markov property was considered in \cite{Komatsu84, Portenko94}.
The existence of a unique string solution for equation \eqref{eq1.1} was proved in \cite{Pilipenko12}.  Besides,
it is continuously dependent on $x:$
$$
\forall \ T>0 \ \forall \ x_0\in\mbR: \ \sup_{t\in[0, T]}|\vf_t(x)-\vf_t(x_0)|\overset{P}{\rightarrow}0, \ x\to x_0.
$$

In this paper we construct a modification of $(\vf_t(x))_{t\geq 0}$ which
is c\'adl\'ag in $t$ and monotonous in $x.$ We prove that if a
function $a$ has a locally bounded variation, then
$(\vf_t(x))_{t\geq0}$ is Sobolev differentiable in $x$ a.s. and the
derivative $\nabla\varphi_t(x):=\frac{\partial \varphi_t(x)}{\partial x}$ has the following representation
\begin{equation}
\label{eq2.1}
\nabla\vf_t(x)=\exp\left\{\int_{\mbR}L^{\vf(x)}_t(y)da(y)\right\},
\end{equation}
where $L^{\vf(x)}_t(y)$ is a local time of the process
$(\vf_s(x))_{s\in[0, t]}$ at the point $y.$

 Formula
\eqref{eq2.1} can be easily explained for $a\in C^1(\mbR).$
Indeed, in this case for each
$\omega$, equation \eqref{eq1.1} can be considered  as an integral equation with continuously differentiable
coefficients. Then $\vf_t(x)$ is
continuously differentiable in $x$ and
$\nabla\vf_t(x)$ satisfies the linear
equation
\begin{equation}
\label{eq3.1}
\nabla\vf_t(x)=1+\int^t_0a'(\vf_s(x))\nabla\vf_s(x)ds.
\end{equation}
So
\begin{equation}
\label{eq3.2}
\nabla\vf_t(x)=\exp\left\{\int^t_0a'(\vf_s(x))ds\right\}.
\end{equation}
By the occupation times formula \cite{Berman85}, the
r.h.s. of \eqref{eq3.2} is equal to
$$
\exp\left\{\int_\mbR
a'(y)L^{\vf(x)}_t(y)dy\right\}=\exp\left\{\int_\mbR
L^{\vf(x)}_t(y)da(y)\right\} \ \mbox{a.s.}
$$

\begin{remark}\label{remark_0} All the technical details needed for the
existence of local time such as the validity of occupation
times formula, the existence of the integrals etc.
will be given in the next sections.
\end{remark}

To prove \eqref{eq2.1} for $a$ being a function of bounded
variation  we will use an approximation of \eqref{eq1.1} by SDEs
with $C^1$ drifts.

If $(Z(t))_{t\geq0}$ is a Wiener process that corresponds to $\alpha=2$ then the similar
problem is well studied even for non-additive noises (see for
example \cite{Attanasio10, Bouleau+91, Flandoli+10, Kunita90}). Note that
most techniques used in a Wiener case for non-smooth $a$
(Zvonkin's transformation, Tanaka's formula, Girsanov's formula
etc.) are inapplicable to a case of L\'evy process. It is worth note that the differentiability w.r.t. initial data of solutions of SDEs with jumps and non-smooth coefficients has not been studied.

The paper is organized as follows. The results on measurability and continuity of the solution, and the estimates on transition density are represented in Section \ref{section_properties}. In section \ref{section_local_time}  we give a definition of a local time  and prove the existence  of the local time for the process $(\varphi_t(x))_{t\geq0}$. The main result on differentiability of the solution is given in Section \ref{section_main}, Theorem \ref{thm2}.

\section{Properties of solution}\label{section_properties}
In this section we construct a version of $(\vf_t(x))_{t\geq0}$
satisfying some measurability properties.

Put $\cF_t=\sigma\{Z(s): 0\leq s\leq t\}.$

\begin{proposition}
\label{prop1}
Let  $a(x), \  x\in\mathds{R},$ be a bounded measurable function. Then
\begin{enumerate}[1)]
\item  There exists a unique strong solution of \eqref{eq1.1},
i.e. a $\cF_t$-adapted c\'adl\'ag process {$(\vf_t(x))_{t\geq0}$} that satisfies \eqref{eq1.1} almost surely. \item
{The process $(\vf_t(x))_{t\geq0}$ is continuous w.r.t. $x$} in
probability in topology of uniform convergence:
\begin{equation}
\label{eq5.1}
\forall \ T>0 \ \forall \ x_0\in\mbR:  \ \ \
\sup_{t\in[0, T]}|\vf_t(x)-\vf_t(x_0)|\overset{P}{\rightarrow}0, \ x\to x_0.
\end{equation}
\item The process $(\vf_t(x))_{t\geq0}$ is a homogeneous strong
Markov process. It has a continuous transition density $p_t(x,
y).$ Moreover,

$
\forall \ T>0 \ \exists \ N_T=N_{T, \|a\|_{L_p}}  \ \forall \ t\in(0, T] \ \forall \ x\in\mbR,\ y\in\mbR:
$
\begin{equation}
\label{eq6.2}
p_t(x,y)\leq\frac{N_Tt}{(t^{1/\alpha}+|y-x|)^{\alpha+1}},
\end{equation}
where $\alpha\in(1,2)$ is a parameter of a stable process $(Z(t))_{t\geq0}.$

\item If $x_1\leq x_2,$ then
\begin{equation}
\label{eq6.1}
P\{\vf_t(x_1)=\vf_t(x_2), \ t\geq\sigma_{x_1, x_2}\}=1,
\end{equation}
where
$$
\sigma_{x_1, x_2}=\inf\{t\geq0: \ \vf_t(x_1)\geq\vf_t(x_2)\}.
$$
\item The process $(\vf_t(x))_{t\geq0, \ x\in\mbR}$ can be selected
such that
\begin{enumerate}[a)]
\item it is monotonous in $x:$
\begin{equation}
\label{eq6.0}
\forall \ \omega\in\Omega \ \forall \ x_1\leq x_2 \ \forall \ t\geq0: \ \vf_t(x_1, \omega)\leq\vf_t(x_2, \omega);
\end{equation}
\item it is c\'adl\'ag in $x$ for any fixed $t$ and $\omega;$
\item  for any $T>0$ a map
$$
[0, T]\times\mbR\times\Omega\ni(t,x,\omega)\mapsto\vf_t(x, \omega)
$$
is $\cB([0, T])\times\cB(\mbR)\times\cF_T$-measurable.
\end{enumerate}
\end{enumerate}
\end{proposition}
\begin{proof}
For  a proof of 1), 2), see \cite{Pilipenko12}, 3) even in more general case
is proved in \cite{Podolynny+95, Portenko94}.

Prove 4). Since the process
$$
\vf_t(x_2)-\vf_t(x_1)=\int_0^t\left(a(\varphi_s(x_2))-a(\varphi_s(x_1))\right)ds, \ t\geq0,
$$
is continuous in $t,$ then
$$
\sigma_{x_1, x_2}=\inf\{t\geq0: \vf_t(x_1)=\vf_t(x_2)\},
$$
and it is easy to see that  {the} process
$$
\wt{\vf}_t(x_1):=\begin{cases}
\vf_t(x_1), t\leq \sigma_{x_1, x_2},\\
\vf_t(x_2), t>\sigma_{x_1, x_2},
\end{cases}
$$
is a solution of \eqref{eq1.1} with initial value $x_1.$ By the
uniqueness of the solution this implies \eqref{eq6.1}.

We
construct a version  {of the solution} that satisfies
properties of the last part  {of Proposition 1}. Let
$\wt{\Omega}$ be a set of full measure such that \eqref{eq1.1} and
\eqref{eq6.0} are satisfied for all $\omega\in\wt{\Omega},
t\geq0,$ and rational $x.$ The
monotonicity and \eqref{eq5.1} imply that $\wt{\Omega}$ can be selected such that

$
\forall \ \omega\in\wt{\Omega} \ \forall \ x_0\in \mbQ \ \forall \ T>0: $
$$
\sup_{t\in[0, T]}|\vf_t(x, \omega)-\vf_t(x_0, \omega)|\to0, \ x\to x_0, x\in\mbQ.
$$

It is easy to see that
$$
\wt{\vf}_t(x, \omega):=\begin{cases}
\vf_t(x), x\in \mbQ \ \mbox{and} \ \omega\in\wt{\Omega},\\
\lim_{\begin{subarray}{c}
y\downarrow x\\
y\in \mbQ
\end{subarray}}\vf_t(y), x\notin \mbQ  \ \mbox{and} \ \omega\in\wt{\Omega},\\
x, \ \omega\notin\wt{\Omega},
\end{cases}
$$
is a version of $(\vf_t(x))_{t\geq 0, x\in\mbR}$ that satisfies 5).

The Proposition is proved.
\end{proof}

\begin{remark} \label{remark_00} Later on  we will always consider a version of $(\vf_t(x))_{t\geq0}$ that satisfies the assumptions of Proposition \ref{prop1}.
\end{remark}

\section{A local time}\label{section_local_time}
In this section we prove the existence and study properties
of a local time for the process $(\vf_t(x))_{t\geq 0}.$

There are a few different approaches to the notion of a local time. We consider a local time as a density of the occupation measure.  Recall the definition and some properties of local times (see \cite{Berman85, Geman+80}).

Let $X(t), t\geq0,$ be a measurable function. Define a measure
$$
\nu_t(A):=\lambda\{s: X(s)\in A, s\in[0, t]\},
$$
where $\lambda$ is  {a} Lebesgue measure.

\begin{definition}Let the measure $\nu_t$ be absolutely continuous w.r.t. $\lambda.$ Then its Radon--Nikodym derivative $\alpha(y, t)=\frac{d\nu_t(y)}{d\lambda}$ is called a local time of $X$ relative to $[0, t].$
In particular,
$$
\nu_t(A)=\int_A\alpha_t(y)dy, \ A\in\cB(\mbR),
$$
\begin{equation}
\label{eq9.1}
\int^t_0f(X(s))ds=\int_{\mbR}f(y)\nu_t(dy)=\int_{\mbR}f(y)\alpha_t(y)dy,
\end{equation}
where $f$ is a measurable function for which at least one integral in \eqref{eq9.1} make sense.
\end{definition}

\begin{definition}
Let $X(t), t\geq0,$ be a measurable stochastic process. We say
that the local time of $X$ exists a.s. if almost all trajectories have a
local time.
\end{definition}

Assume that for any $0<t_1<t_2<\ldots<t_n$ the distribution of $(X(t_1), \ldots, X(t_n))$ is absolutely continuous.
Let $p(x_1, \ldots, x_n, t_1, \ldots, t_n)$ be the corresponding density.

Put
$$
q_t(x_1, \ldots, x_n)=\int^t_0\ldots\int^t_0p(x_1,\ldots, x_n, t_1, \ldots, t_n)dt_1\ldots dt_n.
$$

\begin{theorem}[See \cite{Berman85}]
\label{B}
If for some $n\geq2$ the function $q_t$ is continuous, then the local time exists a.s., and
\begin{equation}
\label{eq10.1}
E\alpha_t(x_1)\cdot\ldots\cdot\alpha_t(x_n)=q_t(x_1, \ldots, x_n).
\end{equation}
\end{theorem}
\begin{remark}
\label{Remark1}
It was mentioned in the proof (see also \cite{Geman+80}, Sect.25) that
\begin{equation}
\label{eq10.2}
(2\ve)^{-1}\int^t_0\1_{\{|X(s)-y|<\ve\}}ds\to\alpha_t(y), \ \ve\to0+,
\end{equation}
in $L_2$-sense for any $y\in\mbR,$ and almost surely for $\lambda$-a.a. $y$. It follows from \eqref{eq10.2} and standard results on existence of measurable version of a limit (see, for example, \cite{Stricker+78}), that the local time can be selected measurable in $(y, \omega).$ Further we consider only  such a modification.
\end{remark}
\begin{remark}
\label{Remark3}
Note that  if the local time exists a.s., then  \eqref{eq9.1} is satisfied with probability one for any measurable non-negative function $f.$ The exceptional set is independent of $f.$
\end{remark}

Return to equation \eqref{eq1.1}.

\begin{proposition}
\label{prop2} There exists a process $\alpha_x(y, t), x\in\mbR,
y\in\mbR, t\geq0,$ such that

1) for any fixed $x: \alpha_x(y, t)$ is a local time of $\vf_s(x), s\in[0, t];$

2) $\alpha_x(y, t)$ is measurable in $(x, y, \omega, t);$

3) for any $x\in\mbR, t>0,$ a map $y\mapsto\alpha_x(y, t)$ is
continuous in $L_2.$
\end{proposition}
\begin{proof} The existence of the local time for fixed $x$ follows from Theorem \ref{B} and Proposition \ref{prop1}. Indeed, let $n=2,$ then
$$
p(x_1, x_2, t_1, t_2)=p_{t_1}(x, x_1)p_{t_2-t_1}(x_1, x_2)\leq
$$
$$
\leq (N_t)^2t_1^{1-\frac{\alpha+1}{\alpha}}(t_2-t_1)^{1-\frac{\alpha+1}{\alpha}}=
$$
$$
=(N_t)^2t_1^{-\frac{1}{\alpha}}(t_2-t_1)^{-\frac{1}{\alpha}}, 0<t_1<t_2\leq t.
$$
Observe that
$$
\int^t_0\int^{t_2}_0(N_t)^2t_1^{-\frac{1}{\alpha}}(t_2-t_1)^{-\frac{1}{\alpha}}dt_1dt_2=
$$
$$
=(N_t)^2\int^t_0\int^1_0(zt_2)^{-\frac{1}{\alpha}}(t_2-zt_2)^{-\frac{1}{\alpha}}t_2dzdt_2=
$$
\begin{equation}
\label{eq12.1}
=
(N_t)^2\int^t_0t_2^{1-\frac{2}{\alpha}}B\left(1-\frac{1}{\alpha}, 1-\frac{1}{\alpha}\right)dt_2<\infty.
\end{equation}
Above we have used that $0<1-\frac{1}{\alpha}$ and $1-\frac{2}{\alpha}>-1$ because $\alpha\in(1,2).$

The continuity of $q_t$ follows from the Lebesgue dominated convergence theorem.

Remark \ref{Remark1} allows us to select  a measurable in $(x, y, \omega, t)$ modification.

The Proposition is proved.
\end{proof}

\begin{remark} The local time from Proposition \ref{prop2}
coincides with  {that} obtained by N.I.Portenko
\cite{Portenko73}, who considered it as a $W$-functional
from Markov process.
\end{remark}

In the following statement we obtain the exponential integrability of the local time.

\begin{proposition}
\label{prop3}
For any $t>0, \mu>0,$ there exists $c=c(t, \|a\|_\infty, \mu)$ such that
\begin{equation}
\label{eq13.1}
\forall \ x,y\in\mbR \ \  E\exp\{\mu\alpha_x(y, t)\}\leq c.
\end{equation}
\end{proposition}
A possible way to prove \eqref{eq13.1} is to expand the exponent in a
Taylor series, then to use estimate \eqref{eq6.2}, and to make
calculations similar to \eqref{eq12.1} and formula \eqref{eq10.1}.
However it is easier to apply the following result of
N.I.Portenko.

\begin{lemma}
 Assume that $\{\beta(t), t\in[0, T]\}$ is a non-negative measurable process adapted to a flow $\{\cF_t, t\in[0, T]\}.$ Assume that for $0\leq s\leq t\leq T$
$$
E\left\{\int^t_s\beta(\tau)d\tau/\cF_s\right\}\leq\rho(s, t),
$$
where $\rho(s, t)$ is a non-random integral function satisfying the following conditions

a) $\rho(t_1, t_2)\leq\rho(t_3, t_4)$ if $(t_1, t_2)\subset(t_3, t_4);$

b) $\lim_{h\downarrow {0}}\sup_{0\leq s\leq t\leq
 {h}}\rho(s, t)=0.$

Then for any $\lambda$
$$
E\exp\left\{\lambda\int^T_0\beta(\tau)d\tau\right\}\leq c,
$$
where $c$ depends only on $\lambda, T$ and $\rho.$
\end{lemma}
See \cite{Portenko90}, Lemma 1.1 for the proof.

\begin{proof}[Proof of Proposition 3] Let now $\beta_\ve(t)=(2\ve)^{-1}\1_{|\vf_t(x)-y|<\ve}.$ Similarly to \eqref{eq12.1} we obtain that uniformly in $x, y, \ve$
$$
E\left(\int^t_s(2\ve)^{-1}\1_{|\vf_\tau(x)-y|<\ve}d\tau/\cF_s\right)=
$$
$$
=\int^{t-s}_0\int^{y+\ve}_{y-\ve}(2\ve)^{-1}p_r(\vf_s(x), u)dudr\leq
$$
$$
\leq\int^{t-s}_0\frac{N_Tr}{r^{\frac{\alpha+1}{\alpha}}}dr=N_T(t-s)^{\frac{\alpha-1}{\alpha}}\cdot\left(\frac{\alpha}{\alpha-1}\right)=:\rho(s, t), \ 0\leq s\leq t\leq T.
$$
This implies the uniform in $x, y, \ve$ estimate of
$$
E\exp\left\{\lambda\int^t_0(2\ve)^{-1}\1_{|\vf_s(x)-y|<\ve}ds\right\}.
$$
To conclude the proof, it remains to make $\ve\to0$ and apply Fatou's lemma.

Proposition \ref{prop3} is proved.
\end{proof}

\section{Representation of the derivative}\label{section_main}
Consider equation \eqref{eq1.1}. Assume that $a$ is continuously
differentiable. Then for each  $\omega\in\Omega,$
\eqref{eq1.1} can be considered as an integral equation with
$C^1$-coefficients. So $\vf_t(x)$ is differentiable in $x$ and
$\nabla\vf_t(x)=\frac{\pt\vf_t(x)}{\pt x}$ satisfies a linear
equation
$$
\nabla\vf_t(x)=1+\int^t_0a'(\vf_s(x))\nabla\vf_s(x)ds.
$$
Thus
$$
\nabla\vf_t(x)=\exp\left\{\int^t_0a'(\vf_s(x))ds\right\}.
$$
Applying \eqref{eq9.1} and Proposition \ref{prop2} we get
$$
\nabla\vf_t(x)=\exp\left\{\int_\mbR a'(y)\alpha_x(y,
t)dy\right\}=\exp\left\{\int_\mbR \alpha_x(y, t)da(y)\right\} \
\mbox{a.s.}
$$
Note that generally speaking the exceptional set depends  on $x$ and $t.$ By Fubini's theorem
\begin{equation}
\label{eq15.1}
P\left\{\nabla\vf_t(x)=\exp\left\{\int_\mbR\alpha_x(y,
t)da(y)\right\} \ \mbox{for} \ \lambda\mbox{-a.a.} \ x \right\}=1.
\end{equation}

We will justify representation \eqref{eq15.1} for solution of \eqref{eq1.1},
where a function $a$ is not necessarily $C^1,$ but it has a finite variation. We need some definitions and facts on Sobolev spaces.

\begin{definition}
\label{defn1} A function $f: [a, b]\to\mbR$ belongs to a Sobolev
space $W^1_p([a, b]),$ $ p\geq1,$ if $f$ has an absolutely continuous modification and
$\frac{df}{dx}\in L_p([a, b]).$
\end{definition}
Put
$$
\|f\|_{p, 1}:=\|f\|_{W^1_p([a, b])}:=\|f\|_{L_p([a,
b])}+\left\|\frac{df}{dx}\right\|_{L_p([a, b])}.
$$
It is well known that $(W^1_p([a, b]), \|\cdot\|_{p, 1})$ is a Banach space. So if $\{f_n\}\subset W^1_p([a, b])$ is such that
\begin{equation}
\label{eq16.1}
f_n\to f, \ \mbox{and} \ \frac{df_n}{dx}\to g \ \mbox{as} \ n\to\infty \ \mbox{in} \ L_p,
\end{equation}
then
\begin{equation}
\label{eq16.2}
f\in W^1_p([a, b]), \ g=\frac{df}{dx}.
\end{equation}
\begin{definition}
\label{defn2}
A measurable function $f: \mbR\to\mbR$ belongs to the space $W^1_{p, loc}$  if its restriction to any segment $[a, b]$ lies in $W^1_p([a, b]).$
\end{definition}

The main result of the paper is the following theorem.

\begin{theorem}
\label{thm2}
Assume that $a(x), \ x\in\mathds{R},$ is a measurable  bounded  function and its restriction to any interval has a finite variation. Then $\vf_t(\cdot)\in W^1_{p, loc}$ for any $p\geq1$ a.s. and representation \eqref{eq15.1} holds true for any $t>0.$
\end{theorem}

\begin{corollary}
For all
$\left\{x_1,x_2\right\}\subset\mathds{R}, \ x_1\neq x_2,$
\begin{equation}\label{eq_coalescence}
P\left\{\varphi_t(x_1)\neq\varphi_t(x_2), \ t\geq0\right\}=1.
\end{equation}

Relation (\ref{eq_coalescence}) can be obtained similarly to that of \cite{Aryasova+12}, Theorem 1, using the fact that
$$
\varphi_t(x_2)-\varphi_t(x_1)=\int_{x_1}^{x_2}\frac{\partial \varphi_t(y)}{\partial y}dy>0.
$$
\end{corollary}

\begin{proof}[Proof of Theorem \ref{thm2}]
Assume at first that $a$ is a function of bounded variation. Set
$$
a_n(x):=\int_\mbR a(y)g_n(x-y)dy,
$$
where
$$
g_n(x)=n g(nx), \ g\in C^\infty_0(\mbR), \ g\geq0, \ \mbox{and} \ \int_\mathds{R}g(z)dz=1.
$$
Then
$$
\sup_{n,x}|a_n(x)|\leq\|a\|_\infty=\sup_x|a(x)|,
$$
$a_n\in C^\infty(\mbR),$  $a_n(x)\to a(x), n\to\infty$, for all points of continuity of $a$, and
\begin{equation}
\label{eq17.2}
\Var(a_n)\leq\Var(a).
\end{equation}

Let $\vf^n_t(x)$ be a solution of \eqref{eq1.1} with $a_n$ instead
of $a, \ \alpha^n_x(y, t)$   {be} its local time. Then
$$
P\left\{\nabla\vf^n_t(x)=\exp\left\{\int_\mbR\alpha^n_x(y,
t)da_n(y)\right\} \ \mbox{for} \ \lambda\mbox{-a.a} \ x\right\}=1.
$$
It follows from \cite{Pilipenko12} that
$$
\forall \ x \ \forall \ T\geq0: \ \sup_{t\in[0, T]}|\vf^n_t(x)-\vf_t(x)|\overset{P}{\rightarrow}0, \ n\to\infty.
$$

The uniform boundedness of $\{a_n\}$ implies
$$
\forall \ p\geq1: \ \sup_x\sup_{t\in[0, T]}E|\vf^n_t(x)-\vf_t(x)|^p<\infty.
$$
So
$$
E\int^b_a|\vf^n_t(x)-\vf_t(x)|^pdx\to0, \ n\to\infty
$$
for any $t>0, a\leq b.$

Prove that

$\forall \ p\geq0 \ \forall \ a\leq b \ \forall \ t>0:$
$$
E\int^b_a\left|\exp\left\{\int_\mbR\alpha^n_x(y, t)da_n(y)\right\}-\exp\left\{\int_\mbR\alpha_x(y, t)da(y)\right\}\right|^pdx\to0, n\to\infty,
$$
By Proposition \ref{prop3}, \eqref{eq17.2}, and Jensen's inequality, it suffices to check the convergence
$$
\forall \ x, t: \ \int_\mbR\alpha^n_x(y, t)da_n(y)\to\int_\mbR\alpha_x(y, t)da(y), n\to\infty,
$$
in probability or in $L_2-$sense.

Assume at first that the function $a$ has a finite support. Let
$R$ be such that $\supp \ {a} \subset[-R, R], \ \supp \ {a_n}
\subset[-R, R]. $

For simplicity denote $a$ by $a_0$ and $\alpha$ by $\alpha^0$.

Let $-R=y_0<y_1<\ldots< y_m=R$ be a dissection of $[-R, R].$ Then
$$
E\left|\int_{\mbR}\alpha^n_x(y, t)da_n(y)-\int_{\mbR}\alpha_x(y, t)da(y)\right|=E\left|\int^R_{-R}\ldots-\int^R_{-R}\ldots\right|\leq
$$
$$
\leq
\sum_{0\leq j<m}E\left|\int_{y_j}^{y_{j+1}}
\Bigg(\alpha^n_x(y, t)-\frac{\int^{y_{j+1}}_{y_j}\alpha^n_x(z, t)dz}
{\Delta y_j}
\Bigg)da_n(y)\right|+
$$
$$
+
\sum_{0\leq j<m}E\left|\int_{y_j}^{y_{j+1}}
\Bigg(\frac{\int^{y_{j+1}}_{y_j}(\alpha^n_x(z, t)-\alpha_x(z, t))dz}
{\Delta y_j}
\Bigg)da_n(y)\right|+
$$
$$
+
\sum_{0\leq j<m}E\left|\int_{y_j}^{y_{j+1}}
\Bigg(\frac{\int^{y_{j+1}}_{y_j}\alpha_x(z, t)dz}
{\Delta y_j}
\Bigg)(da_n(y)-da(y))\right|+
$$
$$
+ \sum_{0\leq j<m}E\left|\int_{y_j}^{y_{j+1}}
\Bigg(\frac{\int^{y_{j+1}}_{y_j}\alpha_x(z, t)dz} {\Delta
y_j}-\alpha_x(y, t) \Bigg)da(y)\right|\leq
$$
$$
\leq
2\sup_{l\geq0}\max_{0\leq j<m}\sup_{y\in[y_j, y_{j+1}]}
E\Bigg|\alpha^l_x(y, t)-\frac{\int^{y_{j+1}}_{y_j}\alpha^l_x(z, t)dz}
{\Delta y_j}
\Bigg|\Var a_l+
$$
$$
+
\max_{0\leq j<m}
E\Bigg|\int^{y_{j+1}}_{y_j}\alpha^n_x(z, t)dz-\int^{y_{j+1}}_{y_j}\alpha_x(z, t)dz
\Bigg|
\frac{\Var a_n}{\min_{0\leq j<m-1}\Delta y_j}+
$$
\begin{equation}
\label{eq20.1}
+
\max_{0\leq j<m}
E
\frac{\int^{y_{j+1}}_{y_j}\alpha_x(z, t)dz}
{\Delta y_j}
\sum^{m-1}_{k=0}|\Delta_ka_n-\Delta_ka_0|=I_1+I_2+I_3,
\end{equation}
where $\Delta y_j=y_{j+1}-y_j; \ \Delta_ka_n=(a_n(y_{k+1})-a_n(y_k)).$

Estimate each term in the r.h.s. of \eqref{eq20.1}.

   By \eqref{eq10.1} and \eqref{eq6.2}, for any fixed $t\geq 0,\ x\in\mbR,$ the  processes $\alpha^l_x(y, t), y\in[-R, R],$  are equicontinuous in $L_2$ (and so in $L_1$) uniformly in $l\geq0, $  i.e.

$
\forall \ \ve>0 \ \forall \ l\geq0 \ \exists \ \delta_1=\delta_1(\ve)>0 \ \forall \ \{y', y''\}\subset[-R, R], \
|y'-y''|<\delta: $
$$
\sqrt{E(\alpha^l_x(t, y')-\alpha^l_x(t, y''))^2}<\ve.
$$
Hence, if $\max_{0\leq j<m}|y_{j+1}-y_j|<\delta_1$, where $\delta_1=\delta_1\left(\ve/(6\sup_{l\geq0}\Var a_l)\right)$ then the term $I_1$ in \eqref{eq20.1} is less than $\frac{\ve}{3}.$

Consider $I_2.$ By the definition of the local time (see (\ref{eq9.1})):
$$
\int^{y_{j+1}}_{y_j}\alpha^n_x(z, t)dz=\int^t_0\1_{\vf^n_z(x)\in[y_j, y_{j+1}]}dz \ \mbox{a.s.}
$$
Therefore,
\begin{multline*}
E\left|\int^{y_{j+1}}_{y_j}\alpha^n_x(z,
t)dz-\int^{y_{j+1}}_{y_j}\alpha_x(z, t)dz\right|\\
\leq E\int^t_0\left|\1_{\vf^n_z(x)\in[y_j,
y_{j+1}]}-\1_{\vf_z(x)\in[y_j, y_{j+1}]}\right|dz.
\end{multline*}

Taking into account that by \cite{Pilipenko12},
$$
\sup_{z\in[0, t]}|\vf^n_z(x)-\vf_z(x)|\overset{P}{\rightarrow}0, \ n\to\infty,
$$
we get
$$
\1_{\vf_z(x)\notin\{y_j, y_{j+1}\}}\left(\1_{\vf^n_z(x)\in[y_j, y_{j+1}]}-\1_{\vf_z(x)\in[y_j, y_{j+1}]}\right)\overset{P}{\rightarrow}0,\ n\to\infty,
$$
for any $z.$

Since
$$
E\int^t_0\1_{\vf_z(x)\in\{y_j, y_{j+1}\}}dz=\int_{\{y_j\}\cup\{y_{j+1}\}}E\alpha_x(z, t)dz=0,
$$
we have the convergence
$$
\left(\1_{\vf^n_z(x)\in[y_j, y_{j+1}]}-\1_{\vf_z(x)\in[y_j, y_{j+1}]}\right)\overset{P}{\rightarrow}0,\ n\to\infty,
$$
and consequently  the convergence
\begin{equation}
\label{eq21.1}
E\Bigg|\int^{y_{j+1}}_{y_j}\alpha^n_x(z, t)dz-\int^{y_{j+1}}_{y_j}\alpha_x(z, t)dz
\Bigg|\to0, \ n\to\infty.
\end{equation}

Recall that
\begin{equation}
\label{eq22.2}
a_n(y)\to a_0(y), \ n\to\infty,
\end{equation}
if $y$ is a point of continuity of $a_0.$ Select a dissection $\{y_k\}$ such that all $\{y_k\}$ are points of continuity of $a_0,$ and $\max_j\Delta y_j<\delta_1.$ Use \eqref{eq21.1} and \eqref{eq22.2} and select $n_0$ such that for any $n\geq n_0:$
$$
\frac{\sup_{p\geq0}\Var a_p}{\min_{0\leq j<m}\Delta y_j}\cdot\max_{0\leq j<m}
E\Bigg|\int^{y_{j+1}}_{y_j}\alpha^n_x(z, t)dz-\int^{y_{j+1}}_{y_j}\alpha_x(z, t)dz
\Bigg|<\frac{\ve}{3}
$$
and
$$
\sup_{-R\leq z\leq R}
E\alpha_x(z, t)\cdot\sum^{m-1}_{k=0}|\Delta_ka_n-\Delta_ka_0|<\frac{\ve}{3}.
$$
So the r.h.s. of \eqref{eq20.1} is less than $\ve$ and the theorem is proved for finite $a.$

Let now $a$ be an arbitrary function that satisfies conditions of Theorem \ref{thm2}.

Let $g\in C^\infty_0(\mbR); g(x)=1,\ |x|\leq1.$ Put $g_n(x)=g(x/n), a_n(x)=g_n(x)a(x).$ Let $\vf^n_t(x)$ be a solution of \eqref{eq1.1} with a drift coefficient equal to $a_n.$

Observe that by uniqueness of the solution we have the equality
\begin{equation}
\label{eq23.1}
\vf^n_t(x)=\vf_t(x)
\end{equation}
for a.a. $\omega$ from the event $\{\sup_{z\in[0, t]}|\vf_z(x)|\leq n\}.$

Let $[c, d]$ be an arbitrary interval. Denote $n_0=n_0(\omega)=\max_{z\in[0, t]}(|\vf_z(c)|+|\vf_z(d)|).$ Making use of \eqref{eq23.1}, Proposition \ref{prop1}, and Fubini's theorem we obtain the equality $\vf_t(x)=\vf^n_t(x)$ valid for all $n\geq n_0$, a.a. $\omega,$ and $\lambda$-a.a. $x\in[c, d].$ Since $a_n$ is finite, $\vf^n_t(\cdot)\in W^1_p([c, d])$ a.s. Thus $\vf_t(\cdot)\in W^1_p([c,d])$ a.s. and its derivatives coincide a.s. with that of $\vf^n_t$ if $n\geq n_0.$  The definition of the local time entails that $\alpha^n_x(y, t)=\alpha_x(y, t),\ n\geq n_0,$  for $\lambda$-a.a. $y\in[c,d]$ with probability 1. So formula \eqref{eq15.1} holds true.

Theorem \ref{thm2} is proved.
\end{proof}


\begin{thebibliography}{10}

\bibitem{Aryasova+12}
O.~V. Aryasova and A.~Yu. Pilipenko.
\newblock On properties of a flow generated by an sde with discontinuous drift.
\newblock {\em Electron. J. Probab.}, 17:no. 106, 1--20, 2012.

\bibitem{Attanasio10}
S.~Attanasio.
\newblock Stochastic flows of diffeomorphisms for one-dimensional {SDE} with
  discontinuous drift.
\newblock {\em Electron. Commun. Probab.}, 15:no. 20, 213--226, 2010.

\bibitem{Berman85}
S.~M. Berman.
\newblock Joint continuity of the local times of markov processes.
\newblock {\em Zeitschrift f{\"u}r Wahrscheinlichkeitstheorie und Verwandte
  Gebiete}, 69:37--46, 1985.

\bibitem{Bouleau+91}
N.~Bouleau and F.~Hirsch.
\newblock {\em Dirichlet forms and analysis on Wiener space}.
\newblock De Gruyter studies in mathematics. W. de Gruyter, 1991.

\bibitem{Flandoli+10}
F.~Flandoli, M.~Gubinelli, and E.~Priola.
\newblock Flow of diffeomorphisms for {SDE}s with unbounded {H}{\"{o}}lder
  continuous drift.
\newblock {\em Bulletin des Sciences Mathematiques}, 134(4):405 -- 422, 2010.

\bibitem{Geman+80}
D.~Geman and J.~Horowitz.
\newblock Occupation densities.
\newblock {\em Ann. Probab.}, 8:1--67, 1980.

\bibitem{Komatsu84}
T.~Komatsu.
\newblock On the martingale problem for generators of stable processes with
  perturbations.
\newblock {\em Osaka J. Math}, (21):113--132, 1984.

\bibitem{Kunita90}
H.~Kunita.
\newblock {\em Stochastic Flows and Stochastic Differential Equations}.
\newblock Cambridge Univ. Press, 1990.

\bibitem{Pilipenko12}
A.~Yu. Pilipenko.
\newblock On existence and properties of strong solutions of one-dimensional
  stochastic equations with an additive noise.
\newblock {\em Theory of Stochastic Processes}.
\newblock (In print). arXiv:1306.0212v1 [math.PR].

\bibitem{Podolynny+95}
S.~I. Podolynny and N.~I. Portenko.
\newblock On multidimensional stable processes with locally unbounded drift.
\newblock {\em Random Operators and Stoch. Equat.}, 3(2):113--124, 1995.

\bibitem{Portenko73}
N.~I. Portenko.
\newblock Non-negative additive functionals of a {M}arkov process and some
  limit theorems (in {R}ussian).
\newblock {\em Teor. Sluch. Protsess.}, (1):86--107, 1973.

\bibitem{Portenko90}
N.~I. Portenko.
\newblock {\em Generalized diffusion processes}.
\newblock Translations of mathematical monographs. American Mathematical
  Society, 1990.

\bibitem{Portenko94}
N.~I. Portenko.
\newblock Some perturbations of drift-type for symmetric stable processes.
\newblock {\em Random Operators and Stoch. Equat.}, 2(3):211--224, 1994.

\bibitem{Stricker+78}
C.~Stricker and M.~Yor.
\newblock Calcul stochastique d\'ependant d'un param\`etre.
\newblock {\em Z. Wahrscheinlichkeitstheor. Verw. Geb.}, 45:109--133, 1978.

\end{thebibliography}

\end{document}